\documentclass[11pt]{article}

\usepackage{geometry}
\usepackage{color}
\usepackage{float}
\usepackage{textcomp}
\usepackage{amsthm}
\usepackage{amsmath}
\usepackage{amssymb}

\usepackage{changes}

\usepackage{mathtools}

\newtheorem{theorem}{Theorem}[section]
\newtheorem{lemma}[theorem]{Lemma}

\newtheorem{proposition}[theorem]{Proposition}

\DeclareMathOperator{\argmin}{argmin}

\newcommand{\R}{\mathbb{R}}

\newcommand{\inner}[2]{\langle{#1},{#2}\rangle}

\newcommand{\norm}[1]{\|#1\|}

\newcommand{\tos}{\rightrightarrows} 

\newcommand{\X}{\mathcal{X}}
\newcommand{\Z}{\mathcal{Z}}

\newcommand{\M}{\mathcal{M}}

\newcommand{\vgap}{\vspace{.1in}}

\newcommand{\tz}{\tilde z}

\newcommand{\bi}{\begin{itemize}}
\newcommand{\ei}{\end{itemize}}
\newcommand{\ba}{\begin{array}}
\newcommand{\ea}{\end{array}}

\begin{document}

\title{ %On the pointwise iteration-complexity of a dynamic regularized  alternating direction method of multipliers
%with larger stepsize
On the pointwise iteration-complexity of a dynamic regularized  ADMM
with over-relaxation stepsize
%On the stepsize of  a dynamic regularized  alternating direction method of multipliers
%with over-relaxation parameter
%On the $\mathcal{O}(\rho^{-1}\log(\rho^{-1}))$ pointwise iteration-complexity of a dynamic regularized  alternating direction %method of multipliers with larger stepsize
}

\author{
    M.L.N. Gon\c calves
    \thanks{IME/UFG- Caixa
    Postal 131, CEP 74001-970, Goi\^ania-GO, Brazil. (E-mail: {\tt
       maxlng@ufg.br}).  The  work of this author was
    supported in part by  CNPq Grants 406250/2013-8, 444134/2014-0 and 309370/2014-0.}
}

%\date{January 4, 2016}

\maketitle

\begin{abstract}
In this paper, we extend the  improved pointwise iteration-complexity  result of a
dynamic regularized  alternating direction method of multipliers (ADMM) 
for a new  stepsize domain. In this complexity analysis, the stepsize parameter can even be chosen in the interval 
$(0,2)$ instead of  interval $(0,(1+\sqrt{5})/2)$.
As usual, our analysis is established by 
  interpreting  this  ADMM variant as an instance of a  hybrid proximal extragradient framework applied to  a  specific monotone inclusion problem.  
 % It is shown that this  ADMM variant 
%finds a $\rho-$approximate solution  in at most  $ \mathcal{O}\left(\rho^{-1}\log(\rho^{-1})\right)$
%iterations.
%More specifically, we now consider the stepsize domain in the interval 
%$(0,2)$ instead of interval $(0,(1+\sqrt{5})/2)$. This pointwise iteration-complexity result  is established by 
  %interpreting  this  ADMM variant as an instance of a  hybrid proximal extragradient framework applied to  a  specific %monotone inclusion problem.  
%for solving 
%linearly constrained convex problems. With an over-relaxation stepsize domain, we prove that this ADMM variant 
%finds a $\rho-$approximate solution  in at most  $ \mathcal{O}\left(\rho^{-1}\log(\rho^{-1})\right)$
%iterations.
%This pointwise iteration-complexity result  is established by 
 % interpreting%  this  ADMM variant as an instance of a  hybrid proximal extragradient framework applied to  a  specific %monotone inclusion problem.
\\
%   \bigskip 

 \end{abstract}

%    (eck 2 op)
%
%%%%%%%%%
%
% 90-XX Operations research, mathematical programming
%
%   90Cxx Mathematical programming
%
%     90C60 Abstract computational complexity for mathematical programming
%     problems
%
%     90C25 Convex programming
%     90C30 nonlinear programming
%     
%
%%%%%%%%%%
%
%%%%%%%%%
% 47-XX Operator theory
% 
%   47Hxx  Nonlinear operators and their properties
%
%     47H05 Monotone operators 
%
%   47Jxx Equations and inequalities involving nonlinear operators [See also
%       46Txx] {For global and geometric aspects, 47H05 Monotone operators
% 
%     47J20 Variational and other types of inequalities involving
%         nonlinear operators (general) (mont 2)
%     
%      47J22 Variational and other types of inclusions
%
%     47J25  Methods for solving nonlinear operator equations (general)
%
%   47Nxx Miscellaneous applications of operator theory
%
%    47N10 Applications in optimization, convex analysis, 
%          mathematical programming, economics
%%%%%%%%%%
%
% 49-XX
% Calculus of variations and optimal control; optimization
%
%   49Mxx
%   Methods of successive approximations
%
%     49M27 Decomposition methods
%%%%%%%%%%
%
% 65-XX Numerical analysis
%
%    65Xxx Mathematical programming, optimization and variational techniques
%
%      65K05 Mathematical programming {Algorithms; for theory see 90Cxx}
%
%      65K10 Optimization and variational techniques [See also 49Mxx, 93B40]
%
%    65Jxx Numerical analysis in abstract spaces
%
%      65J20 Improperly posed problems; regularization
%%%%%%%%%%

\pagestyle{plain}

\section{Introduction} \label{sec:int}
We are interested in the following  linearly constrained convex problem 
\begin{equation} \label{optl}
\min \{ f(x) + g(y) : A x + B y =b, \; x\in \R^n, y \in \R^p \}
\end{equation}
where $f: \R^{n} \to \R$ and $g:\R^{p} \to \R$ are  convex functions,  $A\in \R^{m \times n}$,  $B \in \R^{m\times p}$ and  $b \in \R^m$.
We assume that  the solution set of \eqref{optl} is nonempty. 
Convex optimization problems with a separable structure such as \eqref{optl}
 appear in many applications areas such as machine learning, compressive sensing and image processing.
The augmented  Lagrangian method (see, e.g., \cite{Ber1}) attempts to solve~\eqref{optl} directly without taking into account its particular structure. 
To overcome this drawback, a variant of the augmented Lagrangian method, namely, the alternating direction method of multipliers (ADMM),  was proposed and studied in \cite{0352.65034,0368.65053}. 
The ADMM takes full advantage of the special structure of the problem by considering each variable separably in an
alternating form  and coupling them into 
the Lagrange multiplier updating; for detailed reviews, see  \cite{Boyd:2011,glowinski1984}. 

Recently, several variants of the  ADMM for solving \eqref{optl}  have been proposed in the literature; see, for example, \cite{Cui,Deng1,GADMM2015,MJR,Gu2015,Hager,He2017,HeLinear,He2015,Lin,PauloRoberto}. 
A dynamic regularized  ADMM (DR-ADMM) with stepsize  $\theta \in (0,(1+\sqrt{5})/2)$ was proposed by 
Gon\c calves at al. \cite{MJR}  whose the   pointwise iteration-complexity is substantially  better than ones   for the
 ADMMs. More specifically, for given $\rho>0$,
 it was proved in \cite{MJR} that the DR-ADMM 
finds a $\rho$-approximate solution of \eqref{optl}  in at most  $ \mathcal{O}\left(\rho^{-1}\log(\rho^{-1})\right)$
iterations.  Although different criteria are used,  in general the  ADMM and its variants  need $ \mathcal{O}\left(\rho^{-2}\right)$ iterations to find this same approximate solution  (see, e.g., \cite{Cui,Deng1,GADMM2015,Gu2015,Hager,He2017,HeLinear,He2015,Lin,monteiro2010iteration}). 
The main goal of this work is to  extend the  improved  pointwise iteration-complexity result of the DR-ADMM obtained in  \cite{MJR}  for a new  stepsize domain $\theta \in (0,(1-\alpha+\sqrt{\alpha^2+6\alpha+5})/2)$, where  $\alpha $ is a nonnegative proximal factor  associated to  the proximal term added to the second subproblem of the  method (see the DR-ADMM in Section~\ref{sec:amal}).  Since the limit of $(\sqrt{\alpha^2+6\alpha+5}-\alpha)$ as $\alpha$ goes to infinity is  3,  the latter stepsize domain becomes  $ (0,2)$    (resp. $(0,(1+\sqrt{5})/2)$) when 
$\alpha$ is sufficiently large  (resp. $\alpha=0$).
It is worth pointing out that the ADMM with a larger stepsize parameter  can  substantially  improve the performance of the method in many applications (see  \cite{FPST_editor,glowinski1984} for more details).
As in \cite{MJR}, our complexity analysis is  done by rewriting problem~\eqref{optl} as a monotone inclusion problem and by analyzing the  DR-ADMM  in the setting of a  generalized  hybrid proximal extragradient (HPE).
It should  be  mentioned that paper \cite{MJR3} was the first one  to discuss complexity results for the ADMM  with stepsize $\theta \in (0,2)$  
 for solving non-convex  linearly constrained  problems and, subsequently, paper  \cite{He2017} studied  convergence and complexity results for  the ADMM with the same stepsize domain of this paper for the convex case.

% The first paper to discuss complexity results for the ADMM  with stepsize $\theta \in (0,2)$ was \cite{???}, and subsequently 

\vspace{2mm}
\noindent { \bf Notation:} 
The set of real numbers is denoted by $\R$. 
 The set of non-negative real numbers  and the set of positive real numbers are denoted by $\R_+$ and $\R_{++}$, respectively. 
For $t>0$, we let $\log^+(t):=\max\{\log t,0\}$. For a finite-dimensional 
real vector space  $\X$ with inner product   $\inner{\cdot}{\cdot}$, its  induced norm is denoted by
$\|\cdot\|$. Denote by $\M^{\X}_{+}$  the space of
selfadjoint positive semidefinite  linear operators on $\X$. For each $H \in  \M^{\X}_{+}$,  the seminorm
induced by $H$ on $\X$ is defined by $\|\cdot\|_H:=\sqrt{\inner{H(\cdot)}{\cdot}}$.

%%%%%%%
\section{Preliminaries results} \label{sec:bas}

In this section, we present  a dynamic regularized HPE   framework and  its pointwise
iteration-complexity result. This framework is an instance of   one  studied in \cite{MJR}. 

Consider  the monotone inclusion problem (MIP)
\begin{align} \label{eq:in}
 0\in T(z)
\end{align}
where $\Z$ is  a finite-dimensional 
real vector space and $T:\Z\tos \Z$ is a maximal monotone operator
\footnote{An operator $T:\Z\tos \Z$ is said to be   monotone if 
$
\inner{z-z'}{s-s'}\geq 0$, for every $z,z'\in \Z,$ $s\in T(z)$ and $s'\in T(z')$.
Moreover, $T$ is maximal monotone if it is monotone and, additionally, if $S$ is a monotone operator such that $T(z)\subset S(z)$ for every $z\in \Z$  then $T=S$.}.
We assume that 
 the solution set of~\eqref{eq:in}, denoted by $T^{-1}(0)$,  is nonempty. 

%%%%%%%%%%

The  dynamic regularized HPE framework attempts to solve the inclusion \eqref{eq:in} by solving approximately a sequence of regularized MIP of the following form
\begin{align} \label{eq:inre}
 0\in T(z)+\mu M(z-z_0)
\end{align}
where  $z_0\in \Z$,  $\mu>0$ and  $M \in \M^{\Z}_{+} $ are fixed.
%%%%%%%%
 % for our convergence analysis, it is essential to
  We also assume that the solution set of~\eqref{eq:inre}
\begin{equation}\label{defZu}
\bar{Z}_{\mu}(M):=\{z\in \Z: \; 0 \; \in \; T(z)+\mu M(z-z_0)\}
\end{equation}
  is nonempty for every $\mu>0$. 
   It can be shown that if 
$M$ is positive definite, then the operator $T(\cdot)+\mu M(\cdot-z_0)$
is maximal $\mu$-strongly monotone which in turn implies that the set $\bar{Z}_{\mu}(M)$ is nonempty for every $\mu>0$ (see, e.g., \cite[Corollary~12.44 and Proposition~12.54]{VariaAna}).
Moreover, the following relation between $\bar{Z}_{\mu}(M)$ and $T^{-1}(0)$ holds for every $\mu>0:$ 
\begin{equation}\label{reldist}
 \|z_0-\bar{z}_\mu\|_M \leq   \|z_0-\bar{z}\|_M \quad \forall  \bar{z}_\mu\in \bar{Z}_{\mu}(M), \; \forall \bar{z} \in T^{-1}(0).
\end{equation}
The above relation follows directly from \cite[Lemma~3.1]{MJR} with $(dw)_z(z')= (1/2)\|z'-z\|^2_M$ for every $z,z' \in \Z$.

Next, we present the  dynamic regularized HPE framework for solving  \eqref{eq:in}, which will be used in order   to analyze the ADMM variant of  Section~\ref{sec:amal}.\

\vgap
\vgap
{
\noindent
\fbox{
\begin{minipage}[h]{6.4 in}
{\bf Dynamic regularized HPE (DR-HPE) framework.}
\begin{itemize}
\item[(0)] Let $z_0\in \Z$, $(\eta_0,\sigma,\tau,\rho
)\in \R_{+}\times  [0,1)\times   (0,1)
\times \R_{++}$ and    $M\in \M^{\Z}_{+}$ be given, and set $\mu=1$ and $k=1$;

\item[(1)] find 
$(z_k,\tz_k,\eta_k)\in \Z\times \Z\times \R_+$  
       such that

       \begin{equation}\label{eq:ec.2}
  M(z_{k-1}-{z_k})  \in \left( T(\tz_k)+\mu M ({\tz}_k-{z_0})\right),  
  \end{equation}
      \begin{equation}	\label{eq:es.2}
 \|{z_k}- {\tz}_k\|_M^2 +\eta_k \leq \sigma \|{z_{k-1}} -\tz_{k}\|_M^2+(1-\tau)\eta_{k-1}; 
    \end{equation}
       
\item[(2)] if $\norm{z_{k-1}-{z_k}}_M\leq \rho/2$, then
 go to  step~3; otherwise,  set  $k\leftarrow k+1$ and go to step~1. 

\item[(3)]  compute 
$v_k:=z_{k-1}-{z_k}-\mu (\tilde z_k-z_0)$;
if  $\norm{v_k}_M\leq \rho$,
then stop and output $(\tilde{z},v) \leftarrow (\tilde z_k,v_k)$;
else, set $\mu  \leftarrow \mu/2$ and $k=1$,
and go to step~1. 
\end{itemize}
\noindent
{\bf end}
\end{minipage}
}
\noindent
\\[2mm]
{\bf Remarks.} 1) The DR-HPE framework corresponds to  the  framework~3 in \cite{MJR} with 
$\lambda_k=1$, $\varepsilon_k=0$ and $(dw)_z(z')=(1/2)\|z'-z\|^2_M$ for every $z,z'\in\Z$.
Now,  if $M$ is the identity operator  and  $\eta_k=0$,   it becomes the DR-HPE framework  in  \cite{Maicon} with $\lambda_k=1$ and  $\varepsilon_k=0$.
2)~The scalar $\mu$  plays the role of a regularization parameter which is dinamically adapted in order to control the 
term $M ({\tz}_k-{z_0})$ in~\eqref{eq:ec.2}.
3)   The  DR-HPE framework is a general setting which  does not specify how to obtain $(z_k,\tilde{z}_k,\eta_k)$ as in step~1.  Specific computation of these elements will depend on implementation of particular  instances of the framework and the properties of the operators $T$ and $M$. 
4)  If $M$ is positive definite and $\sigma= \eta_0=0$, then \eqref{eq:es.2} implies that $\eta_k=0$ and $z_k = \tilde z_k$ for every~$k$, and
then \eqref{eq:ec.2} reduces to an iteration of the proximal point method (in the metric $\|\cdot\|_M$) applied to~\eqref{eq:inre}.

The following result gives the pointwise
iteration-complexity bound for the DR-HPE framework.
%%%%%%%%
\begin{theorem}\label{th:main}
 Suppose  that  $1/(1-\sigma)$ and $1/\tau$ are    $\mathcal{O}(1)$. Then, the DR-HPE framework
  finds a  pair
$(\tilde{z} , v)$ satisfying
$
Mv \in T(\tilde{z})$ and $\norm{v}_M\leq \rho,
$
in at most
\begin{align*}
 \mathcal{O}\left(\left(1+\frac{\sqrt{d^2+\eta_0}}{\rho}\right)\left[1+
\log^+\left(
\dfrac{\sqrt{d^2+\eta_0}}
{ \rho}\right)
\right]\right)
\end{align*}
iterations, where $
d:=\inf \left\{\|z_0-z\|_M :z \in T^{-1}(0) \right\}.
$
\end{theorem} 
\begin{proof}
First of all,  the DR-HPE framework is a special case of  framework~3 in  \cite{MJR} where
$\lambda_k=1$, $\varepsilon_k=0$ and $(dw)_z(z')=(1/2)\|z'-z\|^2_M$ for every $z,z'\in\Z$. Moreover,
it is easy to see that the distance generating function  $w(\cdot)=(1/2)\|\cdot\|^2_M$ is an $(1,1)$-regular with respect to $(\Z,\|\cdot\|_M)$
in the sense of  \cite[ Definition~2.2]{MJR}. Hence, the proof follows directly from \cite[Theorem 3.3]{MJR} (see also 
first remark after \cite[Theorem 3.3]{MJR}) with $M=m=\lambda=1$, $\varepsilon_k=0$, $d_0=d^2/2$, $\tilde r=Mv$
and by taking into account the following property of the dual semi-norm $\|M(\cdot)\|_M^*=\|\cdot\|_M$ (see \cite[ Proposition~A1]{MJR}).
\end{proof}

%%%%%%%%%%
 \section{DR-ADMM and its pointwise iteration-complexity }\label{sec:amal}
%%%%%%%% 

 In this section, we recall the  DR-ADMM for solving \eqref{optl} and establish its  pointwise iteration-complexity result
 for any stepsize $\theta \in (0,(1-\alpha+\sqrt{\alpha^2+6\alpha+5})/2)$, where
 $\alpha $ is a nonnegative proximal factor  associated to  the proximal term added to the second subproblem of the  method.
% Even for this new stepsize domain, we  show that this ADMM variant  is  a special case of the  DR-HPE framework 
%applied to a specific monotone inclusion problem. As a consequence, its  pointwise iteration-complexity result will follows %from 
%Theorem~\ref{th:main}.

%The method that we are interested in studying in this section can be described as follows:
The  DR-ADMM for solving \eqref{optl}  is described as follows:

\vgap
\noindent
\fbox{
\begin{minipage}[h]{6.4 in}
{\bf Dynamic regularized  ADMM (DR-ADMM).}
\\[2mm]
(0) Let an initial point  $(x_0,y_0,\gamma_0) \in \R^n \times \R^p \times \R^{m}$,  positive parameters $\beta$ and $\theta$,  a tolerance  $\rho>0$, a proximal factor $\alpha\geq 0 $, and  matrices $R\in \M^{\R^n}_{+}$ and $S\in \M^{\R^p}_{+}$  be given, and 
set $\mu=1$ and $k=1$;
\\[2mm]
(1)   set 
$\beta_1:=\beta/(\theta+\mu),$ $\beta_2:=\beta(1+\mu),$ $\hat{x}_{k-1}=( x_{k-1}+\mu x_0)/(1+\mu)$ and $\hat{\gamma}_{k-1}:=(\theta \gamma_{k-1}+\mu \gamma_0)/(\theta+\mu) $
and compute  $x_k \in \R^n$ as
\begin{equation} \label{def:tsk-admm}
x_k\in  \displaystyle{\argmin_{x} \left \{ f(x) - \inner{ \hat{\gamma}_{k-1}}{Ax} +
\frac{\beta_1}{2} \| Ax+ By_{k-1}-b \|^2+\frac{1+\mu}{2}\|x-\hat{x}_{k-1}\|_R^2 \right\}};
\end{equation}
(2)  set 
$
\tilde{\gamma}_k:=\hat{\gamma}_{k-1}-\beta_1(Ax_k+By_{k-1}-b),$ 
$\hat{y}_{k-1}:=({y_{k-1}+\mu y_0})/({1+\mu})$ and $
u_k:=\tilde{\gamma}_{k}+\beta_2 (Ax_k+B\hat{y}_{k-1}-b),
$
and compute $(y_k,\gamma_k) \in \R^p \times \R^{m}$  as
{\small
\begin{equation} \label{def:tyk-admm}
y_k\in  \displaystyle{\argmin_{y}} \left \{ g(y) - \inner{ u_{k}}{By} +
\frac{\beta_2}{2}\left[\| Ax_k+By - b \|^2 + \alpha\| B(y-\hat y_{k-1}) \|^2+ \frac1{\beta}{\| y-\hat y_{k-1} \|_{S}^2}\right]\right\},
\end{equation}
}
\begin{equation}\label{admm:eqxk}
\gamma_k := \gamma_{k-1}-\theta\beta\left[Ax_k+By_k-b+{\mu}(\tilde{\gamma}_k-\gamma_0)/{(\beta\theta)}\right]; 
\end{equation}
(3) If
\begin{equation}\label{innerstop}
\left( \norm{\Delta x_k}_R^2+(1+\alpha)\beta \norm{ B\Delta y_k}^2+\norm{ \Delta y_k}_S^2+ (1/({\beta\theta}))\|\Delta \gamma_k\|^2\right)^{1/2} \leq \rho/2,
\end{equation}
where 
\begin{equation}\label{delta}
\Delta x_k:=x_{k-1}-x_k, \quad \Delta y_k:=y_{k-1}-y_k, \quad \Delta \gamma_k:=\gamma_{k-1}-\gamma_k,
\end{equation}
then go to step~4; else set $k \leftarrow k+1$ and go to~step~1;
%\begin{equation}\label{innerstop}
%\left( (1+\alpha)\beta \norm{ B(y_{k-1}-y_k)}^2+ (1/({\beta\theta}))\|\gamma_{k-1}-\gamma_k\|^2\right)^{1/2} \leq \rho/2,
%\end{equation}
  \\[2mm]
(4)  set  $v^x_k:=   \Delta x_k-\mu  (x_k-x_0),$ $v^y_k:=   \Delta y_k-\mu  (y_k-y_0)$ and $v^{\gamma}_k:=\Delta \gamma_k -{\mu} (\tilde{\gamma}_k-\gamma_0)$; if
\begin{equation}\label{crit2}
\left( \norm{ v^x_k}_R^2+(1+\alpha)\beta \norm{ B v^y_k}^2+\norm{  v^y_k}_S^2+ (1/({\beta\theta}))\| v^{\gamma}_k\|^2\right)^{1/2}  \le \rho,
\end{equation}
%\begin{equation}\label{eq:rxry}
 %v^x_k=   \Delta x_k-\mu  (x_k-x_0), \; v^y_k=   \Delta y_k-\mu  (y_k-y_0), \; v^{\gamma}_k=\Delta \gamma_k -{\mu} %(\tilde{\gamma}_k-\gamma_0),
%\end{equation}
then  stop and output $(x,y,\tilde{\gamma},v^x, v^y,v^{\gamma}) \leftarrow (x_k,y_k,\tilde{\gamma}_k, v^x_k,v^y_k,v^{\gamma}_k)$; otherwise,
set $\mu  \leftarrow \mu/2$ and $k=1$, and go to step~1.
\\[2mm]
\noindent
{\bf end}
\end{minipage}
}
\noindent
\\[2mm]
{\bf Remarks.}  1) The DR-ADMM is equivalent to  the DR-ADMM in \cite{MJR} with an appropriate choice of 
linear operator $G$.
It should be noted, however, that  the complexity result presented there does not establish any relationship between the stepsize $\theta$ and 
proximal term defined by  $G$.
%$G:=\kappa C^*C+S$, where $\kappa\geq 0$. However,  the complexity result presented there does not establish any %relationship between the stepsize $\theta$ and 
%proximal term defined by  $G$.
%we emphasize that the complexity result  obtained here does not follows trivially from the analysis in  \cite{MJR} by chossing %$G$ as above.
   2) As in the DR-HPE framework, the scalar $\mu$ in the DR-ADMM can be seen  as a regularization parameter. 
%Moreover,
%although the parameter  $\mu$ is required to be positive, it is worth mentioning that if $\mu=0$ in the initial step, then the DR-%ADMM  reduces to the  proximal ADMM  with termination criterion \eqref{innerstop} added (in this case,  \eqref{crit2} is %redundant). 
3)  Suitable choices of $R$ and $S$ may becomes the subproblems  \eqref{def:tsk-admm} and \eqref{def:tyk-admm}   easier to solve or even have a closed-form solutions (see   \cite{HeLinear,Wang2012,Yang_linearizedaugmented} for more details).
4) For  convenience, the term ``cycle" will be used to refer to an  execution of steps~1-3 of the DR-ADMM  with a fixed~$\mu$.

%Some comments about the dynamic regularized ADMM are in order. 
%First, although  the above method  requires $\mu>0$,  it is worth pointing out that if $\mu=0$,  
%it becomes  the  ADMM with the termination criterion \eqref{innerstop} added.
%For a fixed $\mu$, an  execution of steps~1-3 of the dynamic regularized M-ADMM  is called  a cycle.
%The implementation of a cycle  has no relevant extra computational cost when compared to the  M-ADMM and, as we will show  later, the %pointwise iteration-complexity bound of the overall method  is substantially better than the  M-ADMM.

In what follows,  we  show that the DR-ADMM  with  $\theta \in (0,(1-\alpha+\sqrt{\alpha^2+6\alpha+5})/2)$  is still a special case of the  DR-HPE framework 
applied to a specific monotone inclusion problem. As a consequence, its  pointwise iteration-complexity result will follows from 
Theorem~\ref{th:main}.

%In what follows, we  show that the DR-ADMM is  an instance of the  DR-HPE framework applied to a specific  monotone %inclusion problem and,  as a by-product, we derive its pointwise iteration-complexity result.

Let us first deduce the aforementioned monotone inclusion problem.
 It is well known that  a  pair $(\bar{x},\bar{y})$  is a solution  of \eqref{optl} and $\bar{\gamma}$ is  an associated Lagrange multiplier if and only if 
  $(\bar{x},\bar{y},\bar{\gamma})$ satisfies
 \[
0 \in   \partial f(\bar{x})- A^*{\bar{\gamma}}, \quad 0 \in \partial g(\bar{y})- B^* {\bar{\gamma}}, \quad A\bar{x}+B\bar{y}=b.
 \]
 Since it is assumed that the solution set of \eqref{optl} is nonempty, 
  the existence of the Lagrange multipliers for problem \eqref{optl} is guaranteed; see, for example, \cite[Corollary~28.2.2]{Rockafellar70}. 
Hence,  we may   solve  \eqref{optl} by means of obtaining  a triple $(\bar{x},\bar{y},\bar{\gamma})$ satisfying the following monotone inclusion problem
\begin{equation} \label{FAB}
0\in T(x,y,\gamma) := \left[ \begin{array}{c} \partial f(x)- A^*{\gamma}\\ \partial g(y)- B^* {\gamma}\\ Ax+By-b
\end{array} \right].
\end{equation}
In order to analyze  the DR-ADMM in the setting of Section~\ref{sec:bas}, 
 consider the vector space $\Z:=\R^{n}\times\R^{p}\times \R^{m}$  and the following linear operator  
\begin{equation}\label{seminorm}
Q:= \left( \begin{array}{ccc} R& 0& 0\\ 0& (1+\alpha)\beta B^*B+S& 0 \\ 0& 0& ({\theta\beta})^{-1}I
\end{array} \right)  :\Z \to \Z
\end{equation}
where $I$ is the $m \times m$ identity operator. 
We assume that the set $\bar{Z}_{\mu}(Q)$ as defined in \eqref{defZu} with $z_0=(x_0,y_0,\lambda_0)$,  $T$ and $Q$  as in \eqref{FAB} and \eqref{seminorm}, respectively,  is nonempty for every $\mu>0$. We mention that this assumption is not restrictive. Indeed, it is easy to see that a triple
 $(x,y,\gamma)\in \bar{Z}_{\mu}(Q)$ if and only if  $(x,y,\gamma)$  satisfies the inclusions
\begin{align*}
&0\in\partial f(x)-A^*{\gamma}+\mu R(x-x_0),  \qquad
0\in  \partial g(y)-B^*{\gamma}+\mu [(1+\alpha)\beta  B^*B(y-y_0)+S(y-y_0)], \\
& 0=Ax+By-b+\mu({\gamma}-{\gamma}_0)(\beta\theta)^{-1},
\end{align*}
which is  equivalent to  the pair $(x,y)$ be a solution and $\gamma$ an associated Lagrange multiplier of the following optimization  problem
\begin{equation*}
\min_{(x,y,u)} \left\{f(x)+g(y)+\frac{\mu}{2}\|\left(x-x_0,y-y_0,u(\theta{\beta}/{\mu})+\gamma_0\right)\|_Q^2: Ax+By+u=b\right\}.  
\end{equation*}
Therefore, any classical condition guaranteeing solution of the above problem implies that $\bar{Z}_\mu(Q)$ is nonempty. For instance,  coerciviness of $f$ and $g$, or  positive definiteness of $R$ and $S$ and injectiveness of $B$ (which is equivalent to $Q$ be definite positive).

The next result  shows that  the  DR-ADMM generates
 a suitable  pair $( z_k,\tilde z_k)$ satisfying the inclusion~\eqref{eq:ec.2}
with  $T$   as  in \eqref{FAB} and $M=Q$, where $Q$ is   as  in  \eqref{seminorm}.
%The proof of this result is similar to the correspondent one \cite[Lemma 4.2]{MJR}, we include it here just for sake of %completeness.

\begin{proposition} \label{pr:aux}
Let   $\{(x_k,y_k,\gamma_k, \tilde{\gamma}_k)\}$ be  the kth iterate of a cycle of the  DR-ADMM and let $\{(\Delta x_k,\Delta y_k,\Delta \gamma_k)\}$ be as in \eqref{delta}.  Then,
\begin{equation} \label{aux.0}
Q \left( \begin{array}{c} \Delta x_k\\ \Delta y_k\\ \Delta \gamma_k
\end{array} \right) \in \left( \begin{array}{c} \partial f(x_k)-A^*\tilde{\gamma}_k\\\partial g(y_k)-B^*\tilde{\gamma}_k\\Ax_k+By_k-b
\end{array} \right)+ \mu Q\left( \begin{array}{c} x_k-x_{0}\\ y_k-y_{0}\\ \tilde{\gamma}_k-{\gamma}_{0}
\end{array} \right)
\end{equation}
where  $Q$  is as in \eqref{seminorm}. As a consequence, 
 $z_{k}=(x_{k},y_{k},\gamma_{k})$  and  $\tilde{z}_k=({x}_k,y_k, \tilde{\gamma}_k)$  satisfy the inclusion~\eqref{eq:ec.2}
with $M=Q$ and  $T$   as  in \eqref{FAB}.
\end{proposition}

\begin{proof}
From the optimality condition for \eqref{def:tsk-admm} and definitions of $\tilde{\gamma}_k$ and $\hat{x}_{k-1}$, we have
\begin{align}\label{aux.1}
0 &\in \partial f(x_k)-A^*\left(\hat{\gamma}_{k-1}-\beta_1(Ax_k+By_{k-1}-b)\right)+(1+\mu)R(x_k-\hat{x}_{k-1})\nonumber\\
&= \partial f(x_k)-A^*\tilde{\gamma}_k+R(x_k-x_{k-1})+\mu R(x_k-x_0).
\end{align}
Now, from the optimality condition for \eqref{def:tyk-admm} and definition of $u_k$, we obtain
\begin{align}
0 &\in \partial g(y_k)-B^*(u_k-\beta_2( Ax_k+By_k-b))+(1+\mu)\alpha\beta B^*B(y_k-\hat y_{k-1})+(\beta_2/\beta)S(y_k-\hat{y}_{k-1})\nonumber\\
&= \partial g(y_k)-B^*\tilde{\gamma}_{k}+[(1+\mu)\alpha\beta+\beta_2] B^*B(y_k-\hat{y}_{k-1})+(\beta_2/\beta)S(y_k-\hat{y}_{k-1})\nonumber \\
&=[(1+\alpha)\beta B^*B+S](y_k-y_{k-1})+\partial g(y_k)-B^*\tilde{\gamma}_k+\mu[(1+\alpha)\beta  B^*B+S](y_k-y_0) \label{aux.2}
\end{align}
where the last equality is due to  definitions of  $\beta_2$ and $\hat{y}_{k-1}$.
On the other hand, definition of ${\gamma}_k$ in \eqref{admm:eqxk} implies that
\[
0=(\gamma_k - \gamma_{k-1})/{(\beta\theta)}+Ax_k+By_k-b+{\mu}(\tilde{\gamma}_k-\gamma_0)/{(\beta\theta)}.
\]
Hence, the inclusion \eqref{aux.0} follows from 
 the last equality, \eqref{aux.1}, \eqref{aux.2} and definitions in \eqref{delta} and \eqref{seminorm}. 

The second part of the proposition follows immediately  from   \eqref{aux.0} and definitions of $z_k$, $\tilde z_k$, $M$ and~$T$.
\end{proof}
%%%%%%%%%%%%%%%%%%%%
%In the rest of our analysis, the following quantity will be needed: 
%\begin{equation}\label{def:d0admm}
%d_0:=\inf \left\{\|(x_0,y_0,\gamma_0)-(x,y,\gamma)\|_Q : (x,y,\gamma) \;\mbox{is solution of}\, \eqref{FAB} \right\}.
%\end{equation}
%%%%%%%%%
The following lemma   describes some important properties 
of the sequences generated  during a cycle  of the  DR-ADMM.
%%%%%%%%%%%%%
%%%%%%%%%%%%%%
\begin{lemma}\label{lem:deltak}
Let   $\{(x_k,y_k,\gamma_k, \tilde{\gamma}_k)\}$ be  the kth iterate of a cycle of the  DR-ADMM and let $\{(\Delta x_k,\Delta y_k,\Delta \gamma_k)\}$ be as in \eqref{delta}. Then, the following statements hold:
\\[2mm]
(a) $\tilde{\gamma}_k-\gamma_{k-1}=-\beta B\Delta y_k-\Delta \gamma_k/\theta$;
\\[2mm]
(b)  if   $k=1$ and $\theta \in[1,2)$, then
\[
\frac{1}{{\theta}}\langle B\Delta y_1,\Delta \gamma_1 \rangle\geq\frac{1}{2}\| \Delta y_1\|_{\alpha\beta B^*B+S}^2- \frac{2{\theta}d_0}{2-\theta}
\]
where $d_0:=\inf \left\{\|(x_0,y_0,\gamma_0)-(x,y,\gamma)\|_Q : (x,y,\gamma) \;\mbox{is solution of}\; \eqref{FAB} \right\}$;
\\[2mm]
(c) if   $k\geq 2$, then  \[2\langle B\Delta y_k,\Delta \gamma_k\rangle\geq 2(1-\theta)\inner{B\Delta y_k}{\Delta \gamma_{k-1}}+\theta \| \Delta y_k\|_{\alpha\beta B^*B+S}^2-\theta \| \Delta y_{k-1}\|_{\alpha\beta B^*B+S}^2.\]
\end{lemma}
\begin{proof}
(a)  Definitions of $\gamma_k$, $\tilde{\gamma}_{k-1}$ and $\beta_1$ in the DR-ADMM imply that 
\begin{align*}
\gamma_k&=\gamma_{k-1}-\mu(\tilde{\gamma}_k-{\gamma}_0)  -\theta\beta(Ax_k+By_{k-1}-b)-\theta\beta B(y_k-y_{k-1}) 
\\[2mm]
&=\gamma_{k-1}-\mu(\tilde{\gamma}_k-{\gamma}_0)+  (\theta+\mu)(\tilde\gamma_k-\hat\gamma_{k-1})-\theta\beta B(y_k-y_{k-1}) 
\\[2mm]
&=(1-\theta)\gamma_{k-1}+\theta\tilde{\gamma}_k -\theta\beta B(y_k-y_{k-1}) 
\end{align*}
where the last equality is due to definition of  $\hat{\gamma}_k$. Hence, item (a) follows by simple calculus and \eqref{delta}.

(b) Let  a point  $\bar{z}_\mu:=(\bar{x}_{\mu},\bar{y}_{\mu},\bar{\gamma}_{\mu})\in \bar{Z}_{\mu}(Q)$ (see the assumption following \eqref{seminorm}) and define 
\begin{equation}\label{def:a34}
\tilde{z}_1=(x_1,y_1,\tilde{\gamma}_1) \quad \mbox{and} \quad z_k=(x_{k},y_{k},\gamma_{k}),  \; k=0,1. 
\end{equation}
Using  \eqref{delta}, the fact that $-2\inner{a}{b}\leq \|a\|^2+\|b\|^2$ $\forall  a,b\in \R^m$, and $\theta\geq1,$ we obtain
\begin{align*}
\frac{1}{2}\|\Delta y_1\|_{\alpha\beta B^*B+S}^2-\frac{1}{{\theta}}\langle B\Delta y_1,\Delta \gamma_1 \rangle &\leq 
\frac{1}{2}\left((1+\alpha)\beta\|B(y_1-y_0)\|^2+\| y_1-y_0\|_{S}^2+\frac{1}{\beta\theta} \|\gamma_1-\gamma_{0}\|^2\right)\\
&\leq
(1+\alpha)\beta\left(\| B(y_{1}-\bar{y}_{\mu})\|^2 +
\| B(y_{0}-\bar{y}_{\mu})\|^2\right)+\| y_1-\bar{y}_{\mu})\|_{S}^2 \\
&+\| y_0-\bar{y}_{\mu})\|_{S}^2+\frac{1}{\beta\theta}\|\gamma_1-\bar{\gamma}_{\mu}\|^2
+\frac{1}{\beta\theta}\| \gamma_0-\bar{\gamma}_{\mu}\|^2
\end{align*}
which, combined with \eqref{seminorm}, yields
\begin{equation}\label{eq:q34}
\frac{1}{2}\|\Delta y_1\|_{\alpha\beta B^*B+S}^2-\frac{1}{{\theta}}\langle B\Delta y_1,\Delta \gamma_1 \rangle \leq 
 \|z_1-\bar{z}_\mu\|^2_{Q} +\|z_0-\bar{z}_\mu\|^2_{Q}.
\end{equation}
 On the other hand, note that 
\begin{equation}\label{a456}
\|z_1-\bar{z}_\mu\|_Q^2=\|z_0-\bar{z}_\mu\|_Q^2+\|z_1-\tilde{z}_1\|_Q^2-\|z_0-\tilde{z}_1\|_Q^2+2 \inner{Q(z_1-{z}_0)}{\tilde{z}_1-\bar{z}_\mu}.
\end{equation}
As $0 \in T(\bar{z}_\mu)+\mu Q(\bar{z}_\mu-z_0)$ and  $Q(z_{0}-{z_1})  \in ( T(\tz_1)+\mu Q ({\tz}_1-{z_0}))$  (see Proposition~\ref{pr:aux} with $k=1$), we have $ \inner{Q(z_1-{z}_0)}{\tilde{z}_1-\bar{z}_\mu}\leq 0$. This inequality together with  \eqref{a456} imply that
\begin{equation}\label{a4562}
\|z_1-\bar{z}_\mu\|_Q^2\leq \|z_0-\bar{z}_\mu\|_Q^2+\|z_1-\tilde{z}_1\|_Q^2-\|z_0-\tilde{z}_1\|_Q^2.
\end{equation}
Now, using the definitions in \eqref{seminorm} and \eqref{def:a34}, we have
\begin{align*}
 \|z_1-\tilde{z}_1\|_Q^2-\|z_0-\tilde{z}_1\|_Q^2&\leq\frac{1}{\beta\theta}\|\gamma_1-\tilde{\gamma}_1\|^2-{\beta}\|B(y_1-y_0)\|^2-\frac{1}{\beta\theta}\|\tilde{\gamma}_1-\gamma_0\|^2\\
& =\frac{(\theta-2)}{\beta\theta^2}\|\gamma_1-{\gamma}_0\|^2-\frac{2}{\theta}\langle B(y_{1}-y_{0}),\gamma_1-\gamma_{0} \rangle-{\beta}\|B(y_1-y_0)\|^2\\
&=\frac{(\theta-1)}{\beta\theta^2}\|\gamma_1-{\gamma}_0\|^2- \left\|B(y_1-y_0)+\frac{\gamma_1-\gamma_{0}}{\theta} \right\|^2, 
\end{align*}
where the first equality  is due to item (a) with $k=1$. Therefore,
\begin{align*}
 \|z_1-\tilde{z}_1\|_Q^2-\|z_0-\tilde{z}_1\|_Q^2&\leq\frac{(\theta-1)}{\beta\theta^2}\|\gamma_1-{\gamma}_0\|^2\leq \frac{2(\theta-1)}{\theta}\left( \frac{\|\gamma_1-\bar{\gamma}_{\mu}\|^2}{\beta\theta}+\frac{\|\gamma_0-\bar{\gamma}_{\mu}\|^2}{\beta\theta}\right)\\
&\leq \frac{2(\theta-1)}{\theta} \left(\|z_0-\bar{z}_\mu\|_Q^2+\|z_1-\bar{z}_\mu\|_Q^2\right)
\end{align*}
where   the second inequality is due to  the fact that $2\inner{a}{b}\leq \|a\|^2+\|b\|^2$ for all $ a,b\in \R^m$, and the last inequality is due to \eqref{seminorm} and definitions of $z_0,z_1$ and $\bar{z}_\mu$.
Hence, combining the last estimative with \eqref{a4562},   we obtain 
$$
\|z_1-\bar{z}_\mu\|_Q^2\leq \frac{\theta}{2-\theta}\left(1+\frac{2(\theta-1)}{\theta}\right)\|z_0-\bar{z}_\mu\|_Q^2=\frac{3\theta-2}{2-\theta}\|z_0-\bar{z}_\mu\|_Q^2.
$$ 
Therefore, statement (b) follows from \eqref{eq:q34}, the last inequality, \eqref{reldist} with $M=Q$, and the definition of $d_0$.

(c)  From  \eqref{aux.0} and definitions    in \eqref{delta} and \eqref{seminorm}, we obtain
$$
 B^*(\tilde{\gamma}_j-(1+\alpha)\beta B(y_j-y_{j-1}))-S(y_j-y_{j-1}) \in \partial g_{\mu,\beta}(y_j)\qquad  \forall  j\geq1,
$$
where $g_{\mu,\beta}(y):=g(y) + (\mu/2)\|y-y_0\|_{(1+\alpha)\beta B^*B+S}^2$  for every $y\in \R^p$. Hence, using item~(a), we have
\[
(1/\theta) B^*(\gamma_j-(1-\theta)\gamma_{j-1})-(\alpha\beta B^* B+S)(y_j-y_{j-1})\in \partial g_{\mu,\beta}(y_j) \qquad \forall j\geq1.
\]
Using \eqref{delta} and the previous inclusion for $j=k-1$ and $j=k$, it follows from the monotonicity of the subdifferential of $g_{\mu,\beta}$  that 
\begin{align*}
0&\leq  \langle  B^*\Delta\gamma_k,\Delta y_{k}\rangle-{(1-\theta)}\langle  B^*\Delta \gamma_{k-1},\Delta y_{k}\rangle-   
\theta\| \Delta y_k\|_{\alpha\beta B^*B+S}^2+\theta\inner{(\alpha\beta B^*B+S)\Delta y_{k-1}}{\Delta y_{k}}
\end{align*}
which, combined with the fact that $2\inner{(\alpha\beta B^*B+S)\Delta y_{k-1}}{\Delta y_{k}}\leq \| \Delta y_k\|_{\alpha\beta B^*B+S}^2+\| \Delta y_{k-1}\|_{\alpha\beta B^*B+S}^2$,   yields item (c).
\end{proof}
%%%%%%%

In the next lemma, we establish a technical result which will be used in order to prove that the DR-ADMM with $\theta \in [1,(1-\alpha+\sqrt{\alpha^2+6\alpha+5})/2)$  is a special case of the 
DR-HPE framework.

%%%%%%%%
\begin{lemma}\label{pro:sigma} 
Assume that  $\theta \in [1,(1-\alpha+\sqrt{\alpha^2+6\alpha+5})/2)$.  Then, there exists a
 parameter  $\bar\tau\in (0,1/2)$ such that  
\begin{equation}
\bar\sigma:=\left(\frac{b+\sqrt{b^2-4ac}}{2a}\right)\in (0,1),
\end{equation}
where $a:=(1-\bar\tau)(1+\alpha)(1+\theta)-\alpha-(1-\theta)^2$, $c:=\left[1-\bar\tau-\alpha \bar\tau(1-\theta)-(1-\theta)^2\right](1-\theta)^2$ and $b:=[(1-\bar\tau)(1+\alpha)(1+\theta)-\alpha-2(1-\theta)](1-\theta)^2-\alpha\bar\tau(1-\theta)+1-\bar\tau.$
%\begin{align*}
%a&:=(1-\bar\tau)(1+\alpha)(1+\theta)-\alpha-(1-\theta)^2, \qquad c:=\left[1-\bar\tau-\alpha \bar\tau(1-\theta)-(1-\theta)^2\right]%(1-\theta)^2,\\
%b&:=[(1-\bar\tau)(1+\alpha)(1+\theta)-\alpha-2(1-\theta)](1-\theta)^2-\alpha\bar\tau(1-\theta)+1-\bar\tau.
%\end{align*}
Moreover,  
\begin{equation}\label{eq:hj34}
\max\left\{(1-\theta)^2,\frac{\bar\tau(\theta-1)}{(1-\bar\tau)\theta-\bar\tau} ,\frac{{1-\bar\tau[1+\alpha(1-\theta)]}}{(1-\tau)(1+\alpha)(1+\theta)-\alpha}\right\}\leq \bar\sigma,
\end{equation}
and  the matrix   
\begin{equation} \label{matrixtheta1>1}
G(\sigma)= \left[
\begin{array}{cc} 
 (1-\bar\tau)[\sigma(1+\theta)-1]+\alpha[{\theta\sigma-\bar\tau(\sigma+ \theta + \sigma \theta -1)}]& (\sigma+\theta-1)(1-\theta) \\[2mm]
(\sigma+\theta-1)(1-\theta)&  \sigma-(1-\theta)^2  \\[2mm]
\end{array} \right]
\end{equation}
is  positive semidefinite for $\sigma=\bar \sigma$.
\end{lemma}
\begin{proof}
First of all,  if $\theta=1$, then $\bar\sigma \in (0,1)$ for any $\bar\tau\in(0,1/2)$. 
Let us now assume that $\theta \in (1,(1-\alpha+\sqrt{\alpha^2+6\alpha+5})/2)$.
Note that, if $\bar\tau=0$, then
\begin{align*}
a=\theta[3-\theta+\alpha]>0, \; b=\theta[(3+\alpha)(1-\theta)^2+2-\theta]>0,
\; a-b+c=\theta^2[1+2\alpha+(1-\alpha)\theta-\theta^2]>0,
\end{align*}
where the last inequality is due to  the fact that $\theta \in (1,(1-\alpha+\sqrt{\alpha^2+6\alpha+5})/2)$. Moreover,
\[b^2-4ac=h(\alpha):=(5+6\alpha+\alpha^2)(1-\theta)^4+2(3+\alpha)(1-\theta)^3-(1+2\alpha)(1-\theta)^2-2(1-\theta)^2+1>0,\]
where the above inequality follows from the fact that  the minimum value of $h$ is greater than zero for any $\theta\in(1,2)$.
Therefore, we conclude that there exists $\bar\tau\in(0,1/2)$ close to $0$ such that
\begin{equation}\label{ae45}
a>0,\quad b>0,  \quad a-b+c>0, \quad b^2-4ac\geq0,
\end{equation}
 which in turn implies $\bar \sigma\in (0,1)$, concluding the proof of the first part of the lemma.

It is a simple algebraic computation to see that $\bar\sigma$ is the largest root of the  second-order equation
$\det(G(\sigma))=0$ and $\det(G(\sigma)) > 0$ for every $\sigma>\bar\sigma$. Moreover, since
$\det(G(\sigma))\leq 0$  for  $\sigma$ equal to   
$(1-\theta)^2$ and  $[{{1-\bar\tau(1+\alpha(1-\theta))}}]/[{(1-\tau)(1+\alpha)(1+\theta)-\alpha}]$, and 
\[{\bar\tau(\theta-1)}/[{(1-\bar\tau)\theta-\bar\tau}]\leq[{{1-\bar\tau(1+\alpha(1-\theta))}}]/[{(1-\tau)(1+\alpha)(1+\theta)-\alpha}]\]
we  obtain    \eqref{eq:hj34} holds. Therefore, since $\det(G(\bar\sigma))=0$, the diagonal entries of $G(\bar\sigma)$ are positive,
and $G(\bar\sigma)$ is symmetric, we conclude that  $G(\bar\sigma)$ is positive semidefinite. 
\end{proof}
%%%%%%%

In next proposition, we will prove that the sequences $\{z_k\}$ and $\{\tilde z_k\}$ as in proposition~\ref{pr:aux}  satisfy the error  condition~\eqref{eq:es.2} with $M=Q$ and appropriate choices of $\tau$, $\sigma$ and $\{\eta_k\}$.

%%%%%%%
\begin{proposition}\label{lemdeltaxy}
Assume that  $\theta \in (0,(1-\alpha+\sqrt{\alpha^2+6\alpha+5})/2)$.
Let   $\{(x_k,y_k,\gamma_k, \tilde{\gamma}_k)\}$ be  the kth iterate of a cycle of the  DR-ADMM and let $\{(\Delta x_k,\Delta y_k,\Delta \gamma_k)\}$ be as in \eqref{delta}. Consider $Q$ and $d_0$ as  in \eqref{seminorm} and  Lemma~\ref{lem:deltak}(b), respectively. Let $\tau$, $\sigma$ and $\{\eta_k\}$ as  
\begin{itemize}
\item[(i)] any $\tau\in(0,1)$,  $\sigma=\theta+(\theta-1)^2$, and  $\eta_k=0$ for all $k\geq0$, if $\theta\in (0,1)$;
\item[(ii)] $\tau=\bar \tau$ and $\sigma=\bar \sigma$, where $\bar \tau$ and $\bar \sigma$  are given by Lemma~\ref{pro:sigma}, and 
\begin{equation}\label{i645}  
 \eta_0= \frac{4(\bar\sigma+\theta-1)d_0}{(2-\theta)(1-\bar\tau)}, \quad \eta_k= \frac{[\bar\sigma-(\theta-1)^2]}{\beta\theta^3}\|\Delta \gamma_k\|^2+\frac{[\bar\sigma+\theta-1]}{\theta(1-\bar\tau)}\|\Delta y_k\|_{\alpha\beta B^*B+S}^2, \;  \forall k\geq 1,
\end{equation}
%otherwise.
if $\theta \in [1,(1-\alpha+\sqrt{\alpha^2+6\alpha+5})/2)$.
 \end{itemize}
 Then,
  $z_{k}=(x_{k},y_{k},\gamma_{k})$,  $\tilde{z}_k=({x}_k,y_k, \tilde{\gamma}_k)$, $\eta_{k-1}$ and $\eta_k$   satisfy the error  condition~\eqref{eq:es.2}
with $M=Q$. 
\end{proposition}
\begin{proof}
Using  definitions of  $z_k$, $\tilde z_k$ and $\Delta y_k$,  and the fact that $M=Q$, we have
\begin{align*}
\sigma\|{z_{k-1}} -\tz_{k}\|_M^2- \|{z_k}- {\tz}_k\|_M^2\geq &   { {(1+\alpha)\sigma\beta \|B\Delta y_{k}\|^2} }+ \sigma\|\Delta y_{k}\|_S^2+\frac{\sigma}{\beta\theta}\|\gamma_{k-1}-\tilde \gamma_k\|^2-\frac{1}{\beta\theta}\|\tilde{\gamma}_k-\gamma_k\|^2,
\end{align*} 
which, combined with \eqref{delta} and Lemma~\ref{lem:deltak}(a), yields
\begin{align}
&\sigma\|{z_{k-1}} -\tz_{k}\|_M^2- \|{z_k}- {\tz}_k\|_M^2\nonumber\\
 &\geq { {(1+\alpha)\sigma\beta \|B\Delta y_{k}\|^2} }+\sigma\|\Delta y_{k}\|_S^2+\frac{\sigma}{\beta\theta}\left\|\beta B\Delta y_{k}+\frac{\Delta {\gamma}_k}{\theta}\right\|^2-\frac{1}{\beta\theta}\left\|\beta B\Delta y_{k}+\frac{(1-\theta)\Delta \gamma_{k}}{\theta}\right\|^2 \nonumber\\
 &{\small =[(1+\alpha)\theta\sigma+{\sigma}-{1}]\frac{\beta\|B\Delta y_{k}\|^2}{\theta} +\sigma\|\Delta y_{k}\|_S^2 +[\sigma-(1-\theta)^2] \frac{\|\Delta {\gamma}_k\|^2}{\beta\theta^3}+\frac{2(\sigma+\theta-1)}{\theta^2}\langle \Delta \gamma_k,B\Delta y_{k}\rangle.} \label{eq:di23}
\end{align} 
If $\theta\in (0,1)$, then the last inequality and $\sigma=\theta+(\theta-1)^2$ imply that 
\[
\sigma\|{z_{k-1}} -\tz_{k}\|_M^2- \|{z_k}- {\tz}_k\|_M^2\geq [\theta+(\theta-1)^2]\| \Delta y_k\|_{\alpha\beta B^*B+S}^2+\left\|\theta\sqrt{\beta} B\Delta y_{k}+\frac{\Delta \gamma_{k}}{\theta\sqrt{\beta}}\right\|^2 \geq 0,
\]
which, combined with definition of $\{\eta_k\}$, proves the desired inequality.

Assume now  that $\theta \in [1,(1-\alpha+\sqrt{\alpha^2+6\alpha+5})/2)$.
Let us  consider two case: $k=1$ and $k>1$.
\\[1mm]
Case 1 ($k=1$):   It follows from  Lemma~\ref{lem:deltak}(b), definition of $\eta_0$ in \eqref{i645}, and $\theta\ge1$
that 
\begin{align*}
\frac{2(\bar{\sigma}+\theta-1)}{\theta^2}\langle B \Delta y_{1},\Delta \gamma_1 \rangle&\geq\frac{(\bar{\sigma}+\theta-1)}{\theta}\left(\alpha\beta\|B\Delta y_{1}\|^2+\|\Delta y_{1}\|_S^2\right)-(1-\bar\tau)\eta_0
\end{align*}
which,  combined with  \eqref{eq:di23} with $k=1$ and  definitions $\sigma$, $\tau$ and  $\eta_1$, yields
\begin{align*}
&\sigma\|{z_{0}} -\tz_{1}\|_M^2- \|{z_1}- {\tz}_1\|_M^2+(1-\tau)\eta_0-\eta_1\\
%& \geq\left[ \bar{\sigma}(1+\theta)-1 \right]\frac{\beta \|B\Delta y_{1}\|^2}{\theta}+\frac{1}{\theta} \left[   \theta\bar{\sigma}-%\frac{\bar\tau(\bar{\sigma}+\theta-1)}{1-\bar\tau} \right]\left(\alpha\beta \|B\Delta y_{1}\|^2+\|\Delta y_{1}\|_S^2\right)\\
& \geq\left[   (1+\alpha)\theta\bar{\sigma}+\bar{\sigma}-1-\frac{\alpha\bar\tau(\bar{\sigma}+\theta-1)}{1-\bar\tau} \right]\frac{\beta}{\theta} \|B\Delta y_{1}\|^2+ \left[   \theta\bar{\sigma}-\frac{\bar\tau(\bar{\sigma}+\theta-1)}{1-\bar\tau} \right]\frac{1}{\theta}\|\Delta y_{1}\|_S^2\geq0
%&=\frac{1}{\theta(1-\bar\tau)}\left[\left[\bar{\sigma}({(1-\bar\tau)(1+\alpha)(1+\theta)-\alpha})- {{1+\bar\tau(1+\alpha(1-\theta))}} %\right]{\beta} \|B\Delta y_{1}\|^2+ \left[ \bar{\sigma} ((1-\bar \tau)\theta-\bar\tau)- \bar\tau(\theta-1) \right]\|\Delta y_{1}\|%_S^2\right].
\end{align*}
where the last inequality is due to inequality \eqref{eq:hj34}. Thus, the error condition ~\eqref{eq:es.2} holds for $k=1$. 
\\[1mm]
Case 2 ($k>1$): Combining   estimate \eqref{eq:di23} with Lemma~\ref{lem:deltak}(c), we have 
\begin{align*}
&\sigma\|{z_{k-1}} -\tz_{k}\|_M^2- \|{z_k}- {\tz}_k\|_M^2
 \geq [(1+\alpha)(\theta\bar\sigma +\bar\sigma+\theta-1)-\theta]\frac{\beta\|B\Delta y_{k}\|^2}{\theta} +[\bar\theta\sigma+\bar\sigma+\theta-1]\frac{ \|\Delta y_{k}\|_S^2}{\theta} \\
 &+[\bar\sigma-(1-\theta)^2] \frac{\|\Delta\gamma_k\|^2}{\beta\theta^3}+\frac{2(1-\theta)(\bar\sigma+\theta-1)}{\theta^2}\inner{B\Delta y_{k}} {\Delta {\gamma_{k-1}}}
 -\frac{(\bar\sigma+\theta-1)}{\theta} \| \Delta y_{k-1}\|_{\alpha\beta B^*B+S}^2.
\end{align*} 
From the last inequality and definition of $\{\eta_k\}$  in \eqref{i645}, we obtain
\begin{align*}
&\sigma\|{z_{k-1}} -\tz_{k}\|_M^2- \|{z_k}- {\tz}_k\|_M^2+(1-\tau)\eta_{k-1}-\eta_k\\
&\geq\left[(1-\bar\tau)(\bar\sigma(1+\theta)-1)+\alpha({\theta\bar\sigma-\bar\tau(\bar\sigma+ \theta + \bar\sigma \theta -1)})\right]\frac{\beta\|B\Delta y_{k}\|^2}{(1-\bar\tau)\theta}
+\left[\bar\theta\bar\sigma-\frac{\bar\tau(\bar\sigma+\theta-1)}{1-\bar\tau}\right]\frac{ \|\Delta y_{k}\|_S^2}{\theta} \\
 &+(1-\bar\tau)[\bar\sigma-(1-\theta)^2] \frac{\|\Delta {\gamma}_{k-1}\|^2}{\beta\theta^3}+\frac{2(1-\theta)(\bar\sigma+\theta-1)}{\theta^2}\inner{B\Delta y_{k}}{\gamma_{k-1}-\Delta {\gamma}_{k-1}}\\
 &=\left[\bar\theta\bar\sigma-\frac{\bar\tau(\bar\sigma+\theta-1)}{1-\bar\tau}\right]\frac{ \|\Delta y_{k}\|_S^2}{\theta}+
 (w_1,w_2)G(\bar\sigma)(w_1,w_2)^*,
 \end{align*} 
 where $G(\bar\sigma)$ is as in  \eqref{matrixtheta1>1},  $w_1=(\sqrt{\beta\theta/(1-\tau)})B \Delta y_{k}$ and $w_2=(\sqrt{(1-\tau)/(\beta\theta)})\Delta \gamma_{k-1}$. Hence, the error  condition~\eqref{eq:es.2}  for $k>1$ now follows from Lemma~\ref{pro:sigma}.
\end{proof}

 We are now ready to prove  the main result of this section.

%%%%%
\begin{theorem}\label{th:maintheoADMM} 
Assume that  $\theta \in (0,(1-\alpha+\sqrt{\alpha^2+6\alpha+5})/2)$ and 
let $Q$ be as in \eqref{seminorm}. Then, 
the DR-ADMM is an instance of the DR-HPE framework for solving problem~\eqref{FAB} 
with inputs  $z_0=(x_0,y_0,\gamma_0)$, $M=Q$,  and parameters $\tau$, $\sigma$ and $\eta_0$  as defined in  Proposition~\ref{lemdeltaxy}. 
As a consequence, it terminates in at most
\begin{equation}\label{eq:iterRADMMtheta<1}
\mathcal{O}\left(\left(1+\frac{{d_0}}{\rho}\right)\left[1 + \log^+\left(\dfrac{{d_0}}{\rho}\right)\right]\right)
\end{equation} 
iterations with $(x,y,\tilde{\gamma},v^x, v^y,v^{\gamma})$  satisfying
\begin{equation}\label{a467}
Q \left( \begin{array}{c} v^x\\ v^y\\ v^{\gamma}
\end{array} \right) \in \left( \begin{array}{c} \partial f(x)-A^*\tilde{\gamma}\\\partial g(y)-B^*\tilde{\gamma}\\Ax+By-b
\end{array} \right)\quad \mbox{and }\quad \|(v^x,v^y,v^{\gamma})\|_Q \le \rho,
\end{equation}
where   $d_0$ is as in Lemma~\ref{lem:deltak}(b).
\end{theorem}
\begin{proof}
Let $\{(x_k,y_k,\gamma_k, \tilde{\gamma}_k)\}$ be the  sequence generated by   a cycle  of  the  DR-ADMM
and  consider the sequences $\{z_k\}$ and $\{\tilde z_k\}$ defined by
\begin{equation}\label{e093}
 z_{k-1}=(x_{k-1},y_{k-1},\gamma_{k-1}), \quad \tilde{z}_k=(x_k,y_k,\tilde{\gamma}_k) , \quad  \forall k \geq1. 
 \end{equation}
It follows from Propositions \ref{pr:aux} and \ref{lemdeltaxy} that the sequences $\{z_k\}$ and $\{\tilde{z}_k\}$  satisfy
inclusion~\eqref{eq:ec.2}  and the error condition~\eqref{eq:es.2} with $T$ as in \eqref{FAB}, $M=Q$, and 
 $\tau$, $\sigma$ and $\{\eta_k\}$ as defined in  Proposition~\ref{lemdeltaxy}. Moreover,  using    $M=Q$ and \eqref{e093}, it is easy to see that  
steps 3 and 4 of the DR-ADMM  correspond to steps 2 and 3 of the DR-HPE framework, respectively.
Therefore, the first statement of the theorem is proved.

Now, since $\eta_0=0$ or $\eta_0=\mathcal{O}(d_0^2)$, 
the second part of the theorem follows  from  the first one and Theorem~\ref{th:main} with $M=Q$, $T$ as  in \eqref{FAB}, $v=(v^x, v^y,v^{\gamma})$, $\tilde{z}=(x,y,\tilde{\gamma})$ and $d=d_0$.
\end{proof}

We end this section by making two remarks. 1)  As already mentioned in Section~\ref{sec:int},  if   $\alpha$ is sufficiently large  (resp. $\alpha=0$), then the stepsize  $\theta$
belong to the interval  $ (0,2)$ (resp. $(0,(1+\sqrt{5})/2)$). 
2) Note that    \eqref{a467}  can be seen as an optimality/feasibility measure of \eqref{optl}.
Indeed, since $Q$ is symmetric  semidefinite positive, if  $\|(v^x,v^y,v^{\gamma})\|_Q=0$, then the left-hand side of the inclusion in~\eqref{a467} is zero,  and hence the pair $(x,y)$ is a solution of \eqref{optl} and $\tilde{\gamma}$ is an associated Lagrange multiplier.

%As mentioned in the introduction,  \eqref{a467} is an optimality/feasibility measure of \eqref{optl}.
%The pointwise iteration-complexity bound \eqref{eq:iterRADMMtheta<1}
 %is considerably  better than the bound
%$\mathcal{O}(d_0^2/\rho^2)$ for the  ADMM established  in \cite{He2015}. 
% ADD THE FOLLOWING COUPLE LINES INTO YOUR PREAMBLE
\let\OLDthebibliography\thebibliography
\renewcommand\thebibliography[1]{
  \OLDthebibliography{#1}
  \setlength{\parskip}{0.3pt}
  \setlength{\itemsep}{0pt plus 0.3ex}
}

%{\setlength{\bibsep}{0.1pt}
{

%\renewcommand{\baselinestretch}{.7}

%\linearspread{1.0}

%\appendix
%\section{Proof of Theorem~\ref{th:ergHPE}}

{
\scriptsize 	%7
%\footnotesize %8
%\small 	%9
\def\cprime{$'$}

}
}

%\bibliographystyle{plain}
%\bibliography{RADMM_ref}

\begin{thebibliography}{10}

\bibitem{Ber1}
D.~P. Bertsekas.
\newblock {\em Constrained optimization and Lagrange multiplier methods}.
\newblock Academic Press, New York, 1982.

\bibitem{Boyd:2011}
S.~Boyd, N.~Parikh, E.~Chu, B.~Peleato, and J.~Eckstein.
\newblock Distributed optimization and statistical learning via the alternating
  direction method of multipliers.
\newblock {\em Found. Trends Mach. Learn.}, 3(1):1--122, 2011.

\bibitem{Cui}
Y.~Cui, X.~Li, D.~Sun, and K.~C. Toh.
\newblock On the convergence properties of a majorized {A}{D}{M}{M} for
  linearly constrained convex optimization problems with coupled objective
  functions.
\newblock {\em J. Optim. Theory Appl.}, 169(3):1013--1041, 2016.

\bibitem{Deng1}
W.~Deng and W.~Yin.
\newblock On the global and linear convergence of the generalized alternating
  direction method of multipliers.
\newblock {\em J. Sci. Comput.}, pages 1--28, 2015.

\bibitem{GADMM2015}
E.~X. Fang, B.~He, H.~Liu, and X.~Yuan.
\newblock Generalized alternating direction method of multipliers: new
  theoretical insights and applications.
\newblock {\em Math. Prog. Comp.}, 7(2):149--187, 2015.

\bibitem{FPST_editor}
M.~Fazel, T.~K. Pong, D.~Sun, and P.~Tseng.
\newblock Hankel matrix rank minimization with applications to system
  identification and realization.
\newblock {\em SIAM J. Matrix Anal. Appl.}, 34(3):946--977, 2013.

\bibitem{0352.65034}
D.~Gabay and B.~Mercier.
\newblock {A dual algorithm for the solution of nonlinear variational problems
  via finite element approximation.}
\newblock {\em Comput. Math. Appl.}, 2:17--40, 1976.

\bibitem{glowinski1984}
R.~Glowinski.
\newblock {\em Numerical Methods for Nonlinear Variational Problems}.
\newblock Springer Series in Computational Physics. Springer-Verlag, 1984.

\bibitem{0368.65053}
R.~Glowinski and A.~Marroco.
\newblock {Sur l'approximation, par \'el\'ements finis d'ordre un, et la
  r\'esolution, par penalisation-dualit\'e, d'une classe de probl\`emes de
  Dirichlet non lin\'eaires.}
\newblock {\em RAIRO Anal. Num\'er.}, 9:41--76, 1975.

\bibitem{MJR3}
M.~L.~N. Gon{\c{c}}alves, J.~G. Melo, and R.~D.~C. Monteiro.
\newblock Convergence rate bounds for a proximal {A}{D}{M}{M} with
  over-relaxation stepsize parameter for solving nonconvex linearly constrained
  problems.
\newblock {\em Avaliable on https://arxiv.org/abs/1702.01850}.

\bibitem{MJR}
M.~L.~N. Gon{\c{c}}alves, J.~G. Melo, and R.~D.~C. Monteiro.
\newblock Improved pointwise iteration-complexity of a regularized {A}{D}{M}{M}
  and of a regularized non-euclidean {H}{P}{E} framework.
\newblock {\em SIAM J. Optim.}, 27(1):379--407, 2017.

\bibitem{Gu2015}
Y.~Gu, B.~Jiang, and H.~Deren.
\newblock A semi-proximal-based strictly contractive {P}eaceman-{R}achford
  splitting method.
\newblock {\em Avaliable on https://arxiv.org/abs/1506.02221}.

\bibitem{Hager}
W.~W. Hager, M.~Yashtini, and H.~Zhang.
\newblock An ${O}(1/k)$ convergence rate for the variable stepsize {B}regman
  operator splitting algorithm.
\newblock {\em SIAM J. Numer. Anal.}, 54(3):1535--1556, 2016.

\bibitem{He2017}
B.~He and F.~Ma.
\newblock Convergence study on the proximal alternating direction method with
  larger step size.
\newblock {\em Avaliable on
  http://www.optimization-online.org/DB\_FILE/2017/02/5856.pdf}.

\bibitem{HeLinear}
B.~He and X.~Yuan.
\newblock On the $\mathcal{O}(1/n)$ convergence rate of the
  {D}ouglas-{R}achford alternating direction method.
\newblock {\em SIAM Journal on Numer. Anal.}, 50(2):700--709, 2012.

\bibitem{He2015}
B.~He and X.~Yuan.
\newblock On non-ergodic convergence rate of {D}ouglas-{R}achford alternating
  direction method of multipliers.
\newblock {\em Numer. Math.}, 130(3):567--577, 2015.

\bibitem{Lin}
T.~Lin, S.~Ma, and S.~Zhang.
\newblock An extragradient-based alternating direction method for convex
  minimization.
\newblock {\em Found. Comput. Math.}, pages 1--25, 2015.

\bibitem{Maicon}
M.~Marques~Alves, R.~D.~C. Monteiro, and B.~F. Svaiter.
\newblock Regularized {H}{P}{E}-type methods for solving monotone inclusions
  with improved pointwise iteration-complexity bounds.
\newblock {\em SIAM J. Optim.}, 26(4):2730--2743, 2016.

\bibitem{monteiro2010iteration}
R.~D.~C. Monteiro and B.~F Svaiter.
\newblock Iteration-complexity of block-decomposition algorithms and the
  alternating direction method of multipliers.
\newblock {\em SIAM J. Optim.}, 23(1):475--507, 2013.

\bibitem{Rockafellar70}
R.~T. Rockafellar.
\newblock {\em Convex Analysis}.
\newblock Princeton University Press, Princeton, 1970.

\bibitem{VariaAna}
R.~T. Rockafellar and R.~J.-B. Wets.
\newblock {\em Variational analysis}.
\newblock Springer, Berlin, 1998.

\bibitem{PauloRoberto}
O.~Sarmiento, E.A.~Papa Quiroz, and P.R. Oliveira.
\newblock A proximal multiplier method for separable convex minimization.
\newblock {\em Optimization}, 65(2):501--537, 2016.

\bibitem{Wang2012}
X.~Wang and X~Yuan.
\newblock The linearized alternating direction method of multipliers for
  dantzig selector.
\newblock {\em SIAM J. Sci. Comput.}, 34(5):2792--2811, 2012.

\bibitem{Yang_linearizedaugmented}
J.~Yang and X.~Yuan.
\newblock Linearized augmented {L}agrangian and alternating direction methods
  for nuclear norm minimization.
\newblock {\em Math. Comput.}, 82(281):301--329, 2013.

\end{thebibliography}

%%%%%%%%
%%%%%%%%%
\end{document}